\documentclass[11pt,reqno]{amsart}
\usepackage{mathtools}
\usepackage{amsfonts}
\usepackage{amsthm}
\usepackage{amsmath}
\usepackage{amssymb}
\usepackage{pictex}
\usepackage{hyperref}

\usepackage[enableskew]{youngtab}

\usepackage{graphicx}
\usepackage{enumitem}

\usepackage{pgf,tikz}
\usetikzlibrary{arrows}
\usetikzlibrary{arrows.meta}
\usetikzlibrary{decorations.pathreplacing}
\usepackage{dynkin-diagrams}

\usepackage{hyperref}
\usepackage{cleveref}
\usepackage{todonotes}

\usepackage[acronym,symbols]{glossaries}

\newacronym{ex}{EX}{example}

\newtheorem{theorem}{Theorem}

\newtheorem{corollary}{Corollary}

\newtheorem{definition}{Definition}

\newtheorem{lemma}{Lemma}

\newcommand{\op}{\operatorname}

\title{The virtual cactus group and Littelmann paths}

\author{Jacinta Torres}
\address{Institute of Mathematics, Jagiellonian University in Krakow}
\email{\href{jacinta.torres@uj.edu.pl}{jacinta.torres@uj.edu.pl}}

\begin{document}

\begin{abstract}
We define a \textit{virtual cactus group} and show that the cactus group action on Littelmann paths is compatible with the virtualization map defined by Pan--Scrimshaw \cite{PS18}. Our \textit{virtual cactus group} generalizes the group with the same name defined  for the symplectic Lie algebra in \cite{AzenhasTarighatTorres22}. 
\end{abstract}


\maketitle

\section{Introduction}
\label{introduction}
Let $\mathfrak{g}$ be a finite dimensional, complex, semisimple Lie algebra. Let $D$ be the Dynkin diagram associated to the root system of
$\mathfrak{g}$, $R$ its root system,   $\Delta=\{\alpha_i:i\in D\} \subset R$ the set of simple roots, $W = W(R)$ its Weyl group, generated by the simple reflections $\{ r_{i}: i \in D
\}$, and $w_{0} \in W$ the longest element of the Weyl group. For a connected sub-diagram $J\subseteq D$, of $D$, denote by $\theta_J:J\rightarrow J$ the unique Dynkin diagram automorphism that satisfies
$\alpha_{\theta_J (j)} = -w_0^J\alpha_j$, for any node $j\in J$, where $w_0^J$ is the longest element of the parabolic subgroup $W^J\subseteq
W$  (the Weyl group for  $\mathfrak{g}$ restricted to $J$) \cite{BB05}. This leads to the following definition by Halacheva.

\begin{definition}\cite{Halacheva2020}
\label{def:cactus}
The \textit{cactus group} $J_{D}$ is the group with generators $s_{J}$, one for each connected  
subdiagram $J$ of $D$, and relations given as follows: 
\begin{itemize}
\item[1.] $s^{2}_{J} = 1$;
\item[2.] $s_{I}s_{J} = s_{J}s_{I}$ for $I,J \subseteq D$ connected subsets if the union $J \cup I$ is disconnected;
\item[3.] $s_{I}s_{J} = s_{\theta_{I}(J)}s_{I}$ if $J \subset I$. 
\end{itemize}

\end{definition}

 \Cref{def:cactus} is a generalization of the original definition of the cactus group defined by Henriques--Kamnitzer in \cite{HenriquesKamnitzer2004}, which was denoted by $J_{n}$ and which corresponds to the cactus group associated to the Dynkin diagram of type $A_{n-1}$.

\subsection{Main results and aim of the paper.}
In this paper we will be concerned with pairs of Dynkin diagrams $(X,Y)$ related by \textit{folding}, that is, there is an injection of sets of nodes $ X \hookrightarrow Y$ which induces an injection of the corresponding Lie algebras  $\mathfrak{g}_X \hookrightarrow \mathfrak{g}_Y$ as described in \cite{BSch17}.  The main result and aim of this paper is the ``virtualization'' of the cactus group $J_{X}$, as defined by Halacheva in \cite{Halacheva2020}, and of its action on $\mathfrak{g}_{X}$-crystals, transferring certain results obtained for the case $C_{n} \hookrightarrow A_{2n-1}$ in \cite{AzenhasTarighatTorres22} to the more general setup described above. This is carried out in \Cref{thm:virtualcacti} and \Cref{thm:commutativediagram}. It consists in defining a group monomorphism $J_{X} \hookrightarrow J_{Y}$ compatible with the action of $J_{X}$ and $J_{Y}$ on $\mathfrak{g}_{X}$, respectively $\mathfrak{g}_{Y}$-crystals. Moreover, by using the virtualization map on Littelmann paths described by Pan--Scrimshaw in \cite{PS18}, instead of the Baker virtualization map used in \cite{AzenhasTarighatTorres22} for Kashiwara--Nakashima tableaux, we obtain a simple rule to compute the partial Sch\"utzenberger--Lusztig involutions of Littelmann paths in $\mathfrak{g}_{X}$-crystals in terms of partial Sch\"utzenberger--Lusztig involutions of Littelmann paths in $\mathfrak{g}_{Y}$-crystals. This is carried out in \Cref{thm:commutativediagram}.

\section{Acknowledgements}
We thank Olga Azenhas and Mojdeh Tarighat Feller for many inspiring conversations, and ICERM for hosting the meeting titled ``Research Community in Algebraic Combinatorics." We especially thank Olga Azenhas for teaching the author a great deal about the cactus group.  The ideas in this paper stemmed from the wish to generalize the results in \cite{AzenhasTarighatTorres22}, which was written as part of the project titled \textit{The A, C, shifted Berenstein--Kirillov groups and cacti}, in the context of the above mentioned meeting. The author was supported by ICERM for this workshop, was also supported by the grant SONATA NCN  UMO-2021/43/D/ST1/02290
 and partially supported by the grant MAESTRO NCN UMO-2019/34/A/ST1/00263.

\section{The cactus group and crystals}
\label{setup}

Let $\Lambda$ be the integral weight lattice and $\Lambda^{+} \subset \Lambda$ be the \textit{dominant weights}. Recall that irreducible finite-dimensional representations of $\mathfrak{g}$ are in one-to-one correspondence with the set of highest weights $\Lambda^{+}$.  We now recall the definition of a semi-normal crystal as in \cite{BSch17}. 

\begin{definition} \label{seminormalcrystal}
A semi-normal $\mathfrak{g}$-crystal consists of a non-empty set $B$ together with maps 

\begin{align*}
\op{wt}:& B \longrightarrow \Lambda \\
e_{i},f_{i}:& B \longrightarrow B \sqcup \left\{ 0\right\}, i \in D
\end{align*}

\noindent such that for all $b,b' \in B$:

\begin{itemize}	
\item $b' = e_i(b)$ if and only if $b = f_i(b')$,
\item if $f_i(b) \neq 0 $ then $\textsf{wt}(f_i(b)) = \textsf{wt}(b)-\alpha_i$;\\
if $e_i(b) \neq 0$, then
$\textsf{wt}(e_i(b)) = \textsf{wt}(b)+\alpha_i$, and
\item  $\varphi_i(b)-\varepsilon_i(b)=  \langle \textsf{wt}(b),\alpha_i^\vee  \rangle$,
\end{itemize}

\noindent where 
\
\begin{align*}
\varepsilon_i(b)&=\max\{a \in \mathbb{Z}_{\geq 0} :e_i^a(b)\neq 0\} \hbox{ and } \\
       \varphi_i(b)&=\max\{a \in \mathbb{Z}_{\geq 0 }:f_i^a(b)\neq 0\}.
       \end{align*}

\noindent 
To each such crystal $B$ is associated a \textit{crystal graph}, a coloured directed graph with vertex set $B$ and edges coloured by elements $i \in D$, where if $f_{i}(b) = b'$ there is an arrow $b \overset{i}{\rightarrow} b'$. We say that a crystal is irreducible if its corresponding crystal graph is connected and finite. 
\end{definition}

The finite irreducible semi-normal $\mathfrak{g}$-crystals are labeled by the dominant weights $\Lambda^{+}$. Given a highest weight $\lambda \in \Lambda^{+}$, the corresponding irreducible crystal is usually denoted by $B(\lambda)$. It encodes important information about the corresponding irreducible finite dimensional representation of $\mathfrak{g}$, $V(\lambda)$. For instance, $\op{dim(V(\lambda))}$ equals the cardinality of $B$, and, in the weight decomposition $V(\lambda) = \underset{\mu \leq \lambda}{\oplus} V(\lambda)_{\mu}$, $\op{dim}(V(\lambda)_{\mu})$ equals the cardinality of the set of $b \in B(\lambda)$ such that $\op{wt}(b) = \mu$. Moreover, for a subinterval $J \subset D$, the crystal corresponding to the Levi restriction of $V(\lambda)$ corresponds to the $\mathfrak{g}_{J}$-crystal $B(\lambda)_{J}$ with crystal graph obtained from the graph for $B(\lambda)$ by deleting edges with labels $i \notin J$. In this paper, we will only deal with crystals whose crystal graphs decompose into connected components, each of which is isomorphic to crystals of the form $B(\lambda)$. These are also known in the literature as \textit{normal} crystals. \\

\subsubsection{Sch\"utzenberger--Lusztig involutions}
There is an elegant internal action of the cactus group $J_{\mathfrak{g}}$ on crystals via partial Sch\"utzenberger--Lusztig involutions, which are generalizations of Sch\"utzenberger--Lusztig involutions originally studied by Berenstein--Kirillov and generalized by Halacheva. For a subinterval $J \subset D$, the partial \break Sch\"utzenberger--Lusztig involution is defined as follows on $B(\lambda)$. Let $v \in B(\lambda)_{J}$ be a \textit{highest weight} element, and let $v_{w^{J}_{0}} \in B(\lambda)_{J}$ be a \textit{lowest weight} element. In particular $\op{wt}(v_{w^{J}_{0}}) = w^{J}_{0}(\op{wt}(v))$ Let $b = f_{i_{r}}\cdots f_{i_{1}}(v)$ for $i_{j} \in J, j \in [1,r]$. Then the partial Sch\"utzenberger--Lusztig involution is the unique involution $\xi_{J} : B(\lambda) \rightarrow B(\lambda)$ which satisfies for each $j \in J$: 
 \
 \
 \begin{align*}
 \xi_{J}(e_{j}(b)) &= f_{\theta_{J}(j)}(\xi_{J}(b))\\
 \xi_{J}(f_{j}(b)) &= e_{\theta_{J}(j)}(\xi_{J}(b))   \hbox{ and } \\
\op{wt}(\xi_{J}(b)) &= w^{J}_{0}(\op{wt}(b)).
\end{align*}
\
\noindent
In fact, $\xi_{J}(b) = e_{\theta_{J}(i_{r})}\cdots e_{\theta_{J}(i_{1})}(v)$. If $J = D$, $\xi_{J}$ is known as the Sch\"utzenberger--Lusztig involution, and denoted simply by $\xi$. Each partial Sch\"utzenberger--Lusztig involution acts as the corresponding Sch\"utzenberger--Lusztig involution applied to each connected component of the Levi-branched crystal $B(\lambda)_{J}$. If our normal crystal $B$ is not connected, partial Sch\"utzenberger--Lusztig involutions are defined in the same way as above, on each connected component. 

\begin{theorem}[Halacheva, \cite{Halacheva2020}]
\label{thm:halacheva}
Let $B$ be a normal $\mathfrak{g}$-crystal. The cactus group $J_{\mathfrak{g}}$ acts on $B$ via partial Sch\"utzenberger--Lusztig involutions, that is, for $J \subset D$ a subinterval, the assignment $s_{J} \mapsto \xi_{J}$ induces a group action. 
\end{theorem}

\section{The virtual cactus group}
Let $X \hookrightarrow Y$ be an embedding of a twisted Dynkin diagram $X$ into a simply-laced Dynkin diagram $Y$ given by folding. More precisely, there is a Dynkin diagram automorphism $\operatorname{aut}: Y \rightarrow Y$ of $Y$ such that there is an edge-preserving bijection $\sigma: X \rightarrow Y/\operatorname{aut}$. The injection of Dynkin diagrams is reflected on the Lie algebras as follows. Let $\mathfrak{g}_{X}$, respectively $\mathfrak{g}_{Y}$ be the complex simple Lie algebras with Dynkin diagram $X$, respectively $Y$. Then the Dynkin diagram automorphism $\op{aut}$ induces a Lie algebra automorphism $\op{aut}: \mathfrak{g}_{Y} \rightarrow \mathfrak{g}_{Y}$. The set of fixed points under this automorphism has the structure of a Lie algebra isomorphic to  $\mathfrak{g}_{X}$ \cite{kac}.  This induces an injection $\mathfrak{g}_{X} \hookrightarrow \mathfrak{g}_{Y}$. Below we list all such pairs, together with the values of $\theta_{X}$ and $\theta_{Y}$. We use the numbering of the vertices given by \cite{BSch17}.

$$
\begin{array}{cccc} 
      \textbf{X} & \textbf{Y} & \textbf{$\theta_{X}$} & \textbf{$\theta_{Y}$}\\

\hline
  		      \textbf{$C_{n}$} & \textbf{$A_{2n-1}$} &  \operatorname{Id}  &   \textbf{$\theta_{Y}(i) = 2n-i$}\\       

\hline 
&&&\\
         \textbf{$B_{2n-1}$} & \textbf{$D_{2n}$} &    \operatorname{Id} & \operatorname{Id}  \\
\hline
         \textbf{$B_{2n}$} & \textbf{$D_{2n+1}$} &  \operatorname{Id} & \theta_{Y}(i) = \begin{cases*} i & if $i < 2n$    \\  2n, 2n+1 & if $i = 2n+1, 2n$ resp. \end{cases*}   \\
\hline
      \textbf{$G_{2}$} & \textbf{$D_{4}$}  &   \operatorname{Id}    & \operatorname{Id} \\
\hline
&&&\\
         \textbf{$F_{4}$} & \textbf{$E_{6}$} &  \operatorname{Id}  &    \theta_{Y}(i)  = \begin{cases*} 6,1 & if $i = 1,6$ resp.   \\  5,3 & if $i = 3,5$ resp.   \\  i & otherwise \end{cases*}  
  \end{array}
  $$

We have $\operatorname{aut} = \theta_{Y}$, except for the cases where $Y= D_{2n}$, where

\[
 \operatorname{aut}(i) =
\begin{cases}
i  & i < 2n-1 \\
2n & i = 2n-1 \\
2n - 1 & i = 2n.
\end{cases}
\]

 We proceed to define a group monomorphism $J_{X}\hookrightarrow J_{Y}$. Its image will be isomorphic to what we call the virtual cactus group, generalizing the concept of the virtual symplectic cactus group defined in \cite{AzenhasTarighatTorres22} for $X = C_{n}$ and $Y = A_{2n-1}$. We start by stating the following lemma, which immediately follows from the description in the previous section. We will abuse notation and consider the coset $\sigma(I) \in Y/\operatorname{aut}$, as a subset of $Y$, for $I \subset X$. Each non-simply laced Dynkin diagram we consider has what we will call in this note a \textit{branching point} $x_{0} \in X$, described in the table below. 

$$
\begin{array}{cccc} 
      \textbf{X} & x_{0}\\
\hline
  		      \textbf{$C_{n}$} & n \\   
\hline
         \textbf{$F_{4}$} & 2\\    

\hline
         \textbf{$B_{n}$} & n-1 \\
\hline
      \textbf{$G_{2}$} & 2 
  \end{array}
  $$

For the comfort of the reader we include the corresponding Dynkin diagrams as well below. 
$$
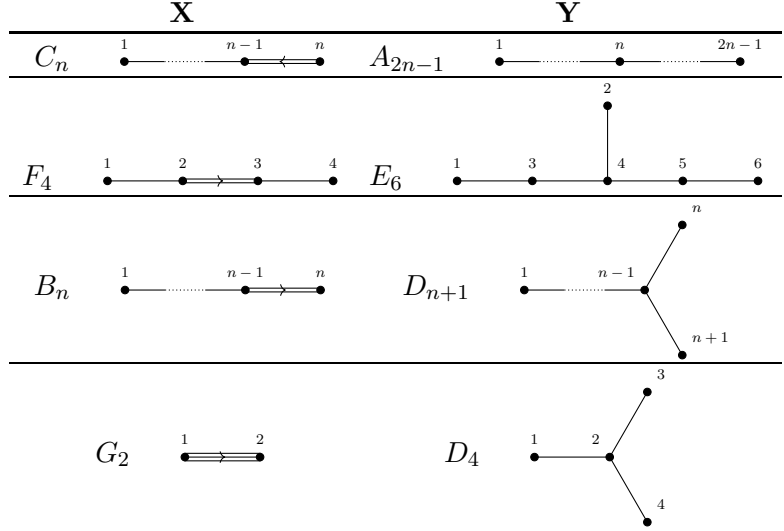
\begin{array}{cccc} 
      \textbf{X} & \textbf{Y}\\
\hline
\scriptsize C_n \qquad  \dynkin[edge length=1cm,labels*={1,n-1,n}] C{*.**} & \scriptsize A_{2n-1} \qquad  \dynkin[edge length=1cm,labels*={1,n,2n-1}] A{*.*.*} \\       

\hline
\scriptsize F_{4} \qquad  \dynkin[edge length=1cm,labels*={1,2,3,4}] F{****} & \scriptsize E_{6} \qquad  \dynkin[edge length=1cm,labels*={1,2,3,4,5,6}] E{******}\\

\hline
\scriptsize B_n \qquad  \dynkin[edge length=1cm,labels*={1,n-1,n}] B{*.**}  & \scriptsize D_{n+1} \qquad  \dynkin[edge length=1cm,labels*={1,n-1,n,n+1}] D{*.***}  \\
\hline
\scriptsize G_{2} \qquad  \dynkin[edge length=1cm,labels*={1,2}] G{**}& \scriptsize D_{4} \qquad  \dynkin[edge length=1cm,labels*={1,2,3,4}] D{****} 
  \end{array}
  $$

We now consider the following elements:

\begin{align*}
\tilde{s}_{I} = \prod s^{Y}_{\tilde{I}} 
\end{align*}

\noindent
where $s^{Y}_{\tilde{I}}$ are the generators of the cactus group $J_{Y}$ and the product is taken over the connected components $\tilde{I}$ of $\sigma(I)$. Our aim for the rest of this section is to prove the following result.

\begin{theorem}
\label{thm:virtualcacti}
The map defined by 
\begin{align*}
\Phi: J_{X} &\rightarrow J_{Y} \\
s_{I} &\mapsto \tilde{s}_{I}
\end{align*}

\noindent is a monomorphism of groups. 
\end{theorem}

\begin{lemma}
\label{lem:thirdvirtualrelation}
Let $I,J \subset X$ such that $J \subset I$. Then 
\begin{align*}
\tilde{s}_{I}\tilde{s}_{J} = \tilde{s}_{\theta_{I}(J)}\tilde{s}_{I}
\end{align*}
\end{lemma}

\begin{proof}
First assume that $\theta_{Y} = \op{Id}$. This means $Y = D_{2n}$ for some $n \geq 2$. If $I = X$ then $\sigma(I) = Y$, therefore the statement of \Cref{lem:thirdvirtualrelation} follows from $\theta_{Y} = \op{Id}$ and the defining Relation 3 for the cactus group $J_{Y}$. If $I \subset X$ does not contain the branching point $x_{0}$ then $\sigma |_I:I \rightarrow \tilde{I} = \sigma(I)$ is an isomorphism, hence the statement follows trivially. If $I$ is not $X$ but contains the branching point, then either $I$ is of type A, $\sigma(I) = \tilde{I}$ is of type A and $\sigma|_{I} : I \rightarrow \tilde{I}$ is an isomorphism, which implies the claim as in the previous case, or $I$ is of type $G_{2}$, in which case the claim also follows easily since $J$ is forced to consist of just one vertex.\\

Assume next that $\theta_{Y} = \operatorname{aut}$. If $I \subset X$ contains the branching point $x_{0}$, then $\theta_{I} = \operatorname{Id}_{I}$ and $\sigma(I) = \tilde{I}$ is connected.  Let us then assume first that $x_{0} \in I$. Now, if $x_{0} \in J$ also, then $\sigma(J) = \tilde{J}$ is connected and $\theta_{\tilde{I}}(\tilde{J}) = \tilde{J}$. Now, if $J \subset I$ does not contain a branching point but $I$ does, then either 

\begin{itemize}
\item $\sigma (J) = \tilde{J}_{1} \sqcup \tilde{J}_{2}$ has two isomorphic connected components, in which case $\theta_{\tilde{I}}(\tilde{J}_{1}) = \tilde{J}_{2}$ and $\theta_{\tilde{I}}(\tilde{J}_{2})) = \tilde{J}_{1}$, or
\item $\sigma(J) = \tilde{J}$ is connected and isomorphic to $J$, in which case $\theta_{\tilde{I}}(\tilde{J}) = \tilde{J}$. 
\end{itemize}

We conclude then that if $x_{0} \in I$ and $\sigma(J) = \tilde{J}$ is connected, then 

\begin{align*}
\tilde{s}_{I} \tilde{s}_{J} = s^{Y}_{\tilde{I}}s^{Y}_{\tilde{J}} = s^{Y}_{\theta_{\tilde{I}}(\tilde{J})}s^{Y}_{\tilde{I}}= s^{Y}_{\tilde{J}}s^{Y}_{\tilde{I}} = \tilde{s}_{J} \tilde{s}_{I} = \tilde{s}_{\theta_{I}(J)} \tilde{s}_{I}, \end{align*}

\noindent as desired. Now, if $x_{0} \in I$ and $\sigma_{J} = \tilde{J}_{1}\sqcup \tilde{J}_{2}$, then we still have $\theta_{I} = \operatorname{Id}$, so $\theta_{I}(J) = J$. We have in this case

\begin{align*}
\tilde{s}_{I} \tilde{s}_{J} = s^{Y}_{\tilde{I}}s^{Y}_{\tilde{J}_{1}}s^{Y}_{\tilde{J}_{2}} = s^{Y}_{\theta_{\tilde{I}}(\tilde{J}_{1})}s^{Y}_{\tilde{I}} s^{Y}_{\tilde{J}_{2}}= s^{Y}_{\theta_{\tilde{I}}(\tilde{J}_{1})} s^{Y}_{\tilde{I} (\tilde{J}_{2})} s^{Y}_{{\tilde{I}}} = \tilde{s}_{J} \tilde{s}_{I} = \tilde{s}_{\theta_{I}(J)} \tilde{s}_{I}.
\end{align*}

\noindent This concludes the proof in the case $x_{0} \in I$.\\

Now let us assume that $x_{0} \notin I$. We have two cases: The case where $\sigma(I)$ is connected is trivial because since $\theta_{Y} = \operatorname{aut}$, we conclude that necessarily $\theta_{\sigma(I)} = \operatorname{aut}|_{\sigma(I)} = \operatorname{Id}_{\sigma(I)}$, also $\sigma(J)\subset \sigma(I)$ is connected for each $J \subset I$, and $\tilde{s}_{J} = s^{Y}_{\sigma(J)}$ for each $J \subset I$. It remains to consider the case where $\sigma(I)$ has two connected components $\sigma(I) = \tilde{I}_{1} \sqcup \tilde{I}_{2}$. It follows that for each $J\subset I$ we have a decomposition into connected components $\sigma(J) = \tilde{J}_{1} \sqcup \tilde{J}_{2}$, where $\tilde{J}_{i} \subset \tilde{I}_{i}, i = 1,2$. The following identity holds by case-by-case analysis:

\begin{align}
\label{localclaim1}
\sigma(\theta_{I}(J)) = \theta_{\tilde{I}_{1}}(\tilde{J}_{1})\sqcup \theta_{\tilde{I}_{2}}(\tilde{J}_{2}).
\end{align}

\noindent Therefore we have in this case: 

\begin{align*}
\tilde{s}_{I}\tilde{s}_{J}  &= s^{Y}_{\tilde{I}_{1}} s^{Y}_{\tilde{I}_{2}} s^{Y}_{\tilde{J}_{1}} s^{Y}_{\tilde{J}_{2}} \\
                                     &= s^{Y}_{\tilde{I}_{1}} s^{Y}_{\tilde{J}_{1}} s^{Y}_{\tilde{I}_{2}} s^{Y}_{\tilde{J}_{2}} \\
                                     & = s^{Y}_{\theta_{\tilde{I}_{1}}(\tilde{J}_{1})}s^{Y}_{\tilde{I}_{1}} s^{Y}_{\theta_{\tilde{I}_{2}}(\tilde{J}_{2})} s^{Y}_{\tilde{I}_{2}} \\
                                     & = s^{Y}_{\theta_{\tilde{I}_{1}}(\tilde{J}_{1})}s^{Y}_{\theta_{\tilde{I}_{2}}(\tilde{J}_{2})}s^{Y}_{\tilde{I}_{1}} s^{Y}_{\tilde{I}_{2}}\\
                                     & = \tilde{s}_{\theta_{I}(J)} \tilde{s}_{I},
 \end{align*}
 
 \noindent where the last equality follows from (\ref{localclaim1}). This concludes the proof in the cases where $\theta_{Y} = \operatorname{aut}$ and therefore the whole proof. 

\end{proof}

\begin{definition}
\label{def:virtualcactusgroup}
The virtual cactus group $J^{v}_{X}$ is defined by generators $s_{\sigma(I)}$, for each $I \subset X$ connected subdiagram, and by the relations: 
\begin{itemize}
\item[1.] $s^{2}_{\sigma(I)} = 1$;
\item[2.] $s_{\sigma(I)}s_{\sigma(J)} = s_{\sigma(J)}s_{\sigma(I)}$  if the union $J \cup I$ is disconnected;
\item[3.] $s_{\sigma(I)}s_{\sigma(J)} = s_{\sigma(\theta_{I}(J))}s_{\sigma(I)}$ if $J \subset I$. 
\end{itemize}
\end{definition}

It is clear from the definition that the virtual cactus group $J^{v}_{X}$ is isomorphic to the cactus group $J_{X}$.

\begin{proof}[Proof of \Cref{thm:virtualcacti}]
To show that $\Phi$ is a group morphism, we need to show three relations: 
\begin{enumerate}
\item $\tilde{s}_{I}^2 = Id$,
\item $\tilde{s}_{I} \tilde{s}_{J} = \tilde{s}_{J} \tilde{s}_{I} $,
\item $\tilde{s}_{I}\tilde{s}_{J} = \tilde{s}_{\theta_{I}(J)}\tilde{s}_{I}$.
\end{enumerate}

Note that the third relation has already been established in \Cref{lem:thirdvirtualrelation}. To prove $(1)$, note that since the connected components of $\sigma(I)$ are disjoint, the commutation relation \emph{2.} in \Cref{def:cactus} implies
\begin{align*}
{\tilde{s}_{I}}^2 = \prod {s^{Y}_{\tilde{I}}}^2 = Id
\end{align*}

\noindent To show the second relation, let $I,J \subset X$ be two disjoint, connected intervals. Then necessarily $\sigma(I)$ and $\sigma(J)$ are mutually disjoint. We have then 

\begin{align*}
\tilde{s}_{I}\tilde{s}_{J} = \prod s^{Y}_{\tilde{I}} \prod s^{Y}_{\tilde{J}}  =  \prod s^{Y}_{\tilde{J}} \prod s^{Y}_{\tilde{I}} 
\end{align*}

\noindent where the third equality follows from relation \emph{2.} for $J_Y$. Note that the image $\Phi(J_X)$ is a group isomorphic to the virtual cactus group $\tilde{J}_{X}$ via the isomorphism $\tilde{s}_{I} \mapsto  s_{\sigma(I)}$. Note that this map is well defined because $\sigma(I) = \sigma(J) \iff I = J$. 
\end{proof}

\section{Virtualization of the action of the cactus group on crystals of Littelmann paths}

In this section we will borrow most of our notation from \cite{PS18} for practical purposes as well as for the comfort of the reader. Let $ \lambda \in \Lambda^{+}.$ We consider $\mathcal{P}(\lambda)$ to be the Littelmann path model for $\lambda$ with paths $\pi: [0,1] \rightarrow \Lambda_{\mathbb{R}}$ of the form 
\
\[\pi(t) = \sum_{i \in D} H_{i,\pi}(t) \Lambda_{i},\]
\noindent
where $H_{i,\pi}(t) = \langle t, \alpha^{\vee}_{i}\rangle$ and where $\Lambda_{i} \in \Lambda^{+}$ are the fundamental weights for $i \in D$. The set $\mathcal{P}(\lambda)$ has the structure of a crystal isomorphic to $B(\lambda)$ with weight map $\op{wt}(\pi) = \pi(1)$. We refer the reader to \cite{PS18} for the definition of the crystal structure using the notation we use in this section. The original and standard reference of the topic is the paper \cite{pathsandrootoperators} by Littelmann. \\

Recall that in this paper we consider embeddings $X \hookrightarrow Y$ given by folding. Let $\Lambda_{X}$ and $\Lambda_{Y}$ be the corresponding integral weight lattices. The bijection $\sigma: X \rightarrow Y/ \op{aut}$ induces a map
\
\[\Psi: \Lambda_{X} \rightarrow \Lambda_{Y}\]
\noindent
given by the assignment 
\
\[\Lambda^{X}_{i} \mapsto \sum_{j \in \sigma(i)} \gamma_{i} (\Lambda^{Y})_{j}, \]

\noindent
where $\gamma_{i}$ is given by Table 5.1 in \cite{BSch17}  and where $\Lambda^{X}_{i}$ and $\Lambda^{Y}_{j}$ denote the fundamental weights in $\Lambda_{X}$, respectively $\Lambda_{Y}$.

\begin{definition}
\label{def:virtuality}
Let $\tilde{B}$ be a normal $\mathfrak{g}_{Y}$-crystal, and a subset $V \subset \tilde{B}$. The \textit{virtual root operators} of type $X$ are, for $i \in X$: 
\begin{align}
\label{virtualoperators}
e^{v}_{i} = \underset{j \in \sigma(i)}{\prod} \tilde{e}^{\gamma_{i}}_{j} \\
f^{v}_{i} = \underset{j \in \sigma(i)}{\prod} \tilde{f}^{\gamma_{i}}_{j}, 
\end{align}

\noindent where $\tilde{e}_{i}, \tilde{f}_{i}, i \in Y$ are the root operators for the $\mathfrak{g}_{Y}$-crystal $\tilde{B}$. 

\noindent A \textit{virtual crystal} is a pair $(V, \tilde{B})$ such that $V$ has a $\mathfrak{g}_{X}$-crystal structure defined by 
\begin{align}
e_{i} := e^{v}_{i} & f_{i} := f^{v}_{i} \\
\varepsilon_{i}:= \gamma^{-1}_{i}\tilde{\varepsilon}_{j} & \varphi_{i}:= \gamma^{-1}_{i}\tilde{\varphi}_{j},
\end{align}

\noindent where $\tilde{\varepsilon}_{j}, \tilde{\varphi}_{j} j \in Y$ denote the maps given by

\begin{align*}
\tilde{\varepsilon}_i(b)&=\max\{a \in \mathbb{Z}_{\geq 0} :\tilde{e}_i^a(b)\neq 0\} \hbox{ and } \\
       \tilde{\varphi}_i(b)&=\max\{a \in \mathbb{Z}_{\geq 0 }:\tilde{f}_i^a(b)\neq 0\}.
       \end{align*}

 If $\mathfrak{g}_{X}$-crystal $B$ is crystal isomorphic to a virtual crystal $V \subset \tilde{B}$ via an isomorphism $\phi: B \rightarrow V$, then the isomorphism $\phi$ is called a \textit{virtualization} map.
\end{definition}

For $\lambda \in \Lambda_{X}^{+}$, the weight $\psi(\lambda) \in \lambda_{Y}$, is dominant, that is, $\psi(\lambda) \in \Lambda^{+}_{Y}$. Given $\pi \in \mathcal{P}(\lambda)$, consider the path $\Psi(\pi):[0,1] \rightarrow \Lambda_{Y}$ defined by 
\
\begin{align}
\label{panscrimshawmap}
\Psi(\pi)(t) = \sum_{i \in D} H_{i,\pi}(t) \psi(\Lambda_{i}) 
\end{align}

\noindent
One of the main results in \cite{PS18} is the following theorem.

\begin{theorem}[Pan--Scrimshaw, \cite{PS18}]
\label{thm:pathvirtualization}
The assignment $\pi \mapsto \Psi(\pi)$ induces a virtualization map
\begin{align*}
\mathcal{P}(\lambda) & \rightarrow \mathcal{P}(\psi(\lambda)) \\
\pi & \mapsto \Psi(\pi). 
\end{align*}
\end{theorem}

The principal aim of this section is to describe the action of the cactus group in terms of the virtualization map of Pan--Scrimshaw. For this, given a connected subdiagram $I \subset X$, let 

\[ \tilde{\xi}_{\sigma(I)} := \prod \xi^{Y}_{\tilde{I}}\]

\noindent
where $\xi^{Y}_{\tilde{I}}$ are the partial Sch\"utzenberger--Lusztig involutions in $\mathcal{P}(\psi(\lambda))$ and the product is taken over the connected components $\tilde{I}$ of $\sigma(I)$. Our next aim is to prove the following result, which generalizes \cite[Theorem 5, Theorem 6, Section 9.5]{AzenhasTarighatTorres22}. 

\begin{theorem}
\label{thm:commutativediagram}
Let $\lambda \in \Lambda_{X}^{+}$ and $\mathcal{P}(\lambda)$ the corresponding Littelmann path model. Then the following diagram commutes 

\[
\vcenter{\hbox{\begin{tikzpicture}[scale=0.79]
\node(tl) at (0,3){$\mathcal{P}(\lambda)$};
\node(tr) at (5,3){$\mathcal{P}(\psi(\lambda))$};
\node(bl) at (0,0){$\mathcal{P}(\lambda)$};
\node(br) at (5,0){$\mathcal{P}(\psi(\lambda))$};
\draw[->](tl)--node[above]{$\Psi$}(tr);
\draw[->](bl)--node[below]{$\Psi$}(br);

\draw[->](tl)--node[left]{$\xi^{X}_{I}$}(bl);
\draw[->](tr)--node[right]{$\tilde{\xi}_{\sigma(I)}$}(br);
\end{tikzpicture}.}}
\]

Moreover, the left inverse $\Psi^{-1}$ can be explicitly computed on $\tilde{\xi}^{Y}_{\sigma(I)}(\Psi(\mathcal{P}(\lambda)))$.
\end{theorem}

\begin{proof}
First note that since the Littelmann  path model $\mathcal{P}(\psi(\lambda))$ is stable under the root operators $\tilde{e}_{i},\tilde{f}_{i}$, it is also stable under the action of the operators $\tilde{\xi}^{Y}_{\sigma(I)}$ for $I \subset X$ connected. Therefore, all paths in $\tilde{\xi}^{Y}_{\sigma(I)}(\Psi(\mathcal{P}(\lambda)))$ must be of the form (\ref{panscrimshawmap}), so the left inverse $\Psi^{-1}$ can be explicitly computed on $\tilde{\xi}^{Y}_{\sigma(I)}(\Psi(\mathcal{P}(\lambda)))$, simply by writing out the corresponding path in this form. We now proceed to show that the diagram commutes.  Let $\pi_{\nu} \in \mathcal{P}(\lambda)_{I}$ be a highest weight path of weight $\op{wt}(\pi_{\nu}) = \pi_{\nu}(1) = \nu$ and  $\pi = f_{i_{r}}\cdots f_{i_{1}}\pi_{\nu}$ for $i_{j} \in I, j \in [1,r]$. Recall that 
\
\[ \xi^{X}_{I}(\pi) = e_{\theta_{I}(i_{r})}\cdots e_{\theta_{I}(i_{1})}\pi_{\nu}.\]

Therefore by \Cref{thm:pathvirtualization} we have 

\begin{align*}
\Psi(\xi^{X}_{I}(b)) = e^{v}_{\theta_{I}(i_{r})}\cdots e^{v}_{\theta_{I}(i_{1})}\Psi(\pi_{\nu}).
\end{align*}

Now, by \Cref{def:virtuality} and \Cref{thm:pathvirtualization} we have 

\begin{align*}
\tilde{\xi}_{\sigma(I)} (\Psi(b))   &=   \prod \xi^{Y}_{\tilde{I}} (\Psi(\pi)) \\
                                                 &=   \prod \xi^{Y}_{\tilde{I}} (\underset{j \in \sigma(i_{r})}{\prod} \tilde{f}^{\gamma_{i_{r}}}_{j} \cdots \underset{j \in \sigma(i_1)}{\prod} \tilde{f}^{\gamma_{i_1}}_{j} (\Psi(\pi_{\nu})))
\end{align*}

\noindent where the product is taken over the connected components $\tilde{I}$ of $\sigma(I)$. To continue our computations we  consider three cases separately: 

\begin{enumerate}
\item The subdiagram  $\sigma(I)  = \tilde{I} \subset Y$ is connected. Then $\theta_{I} = \op{Id}$, we have $\gamma_{i_{j}} = 1$ if and only if $\sigma(i_{j}) = \left\{\tilde{i}^{1}_{j}, \tilde{i}^{2}_{j}\right\}$ or $\sigma(i_{j}) = \left\{\tilde{i}^{1}_{j}, \tilde{i}^{2}_{j}, \tilde{i}^{3}_{j}\right\}$ and $\gamma_{i_{j}} = 2,3$ if and only if $\sigma(i_{j}) = \left\{\tilde{i}_{j}\right\}$. In case $\gamma_{i_{j}} = 1$ we have $\theta_{\tilde{I}}(\tilde{i}^{1}_{j}) = \tilde{i}^{2}_{j} $ and $\theta_{\tilde{I}}(\tilde{i}^{2}_{j}) = \tilde{i}^{1}_{j}$. Moreover, the root operators $\tilde{e}_{\tilde{i}^{1}_{j}}$ and $\tilde{e}_{\tilde{i}^{2}_{j}}$ commute. In case $\gamma_{i_{j}} = 2,3$ we have $\theta_{\tilde{I}}(\tilde{i}_{j}) = \tilde{i}_{j}$. All together this implies: 
 
 \begin{align*}
\tilde{\xi}_{\sigma(I)} (\Psi(b))  &=   \xi^{Y}_{\tilde{I}} ({f}^{v}_{{i_{r}}} \cdots {f}^{v}_{{i_{1}}} (\Psi(\pi_{\nu}))) \\
                                                &=      e^{v}_{\theta_{I}(i_{r})}\cdots e^{v}_{\theta_{I}(i_{1})}\Psi(\pi_{\nu}) \\
                                                &= \Psi(\xi^{X}_{I}(b)).
\end{align*}

\item The subdiagram  $\sigma(I) \subset Y$ is disconnected. Assume $\theta_{Y} = \operatorname{aut}$. In this case we must have $|\sigma(I)| = 2|I|$, that is, $\sigma(I) = \tilde{I}_{1} \sqcup \tilde{I}_{2}$ is a disconnected union. In particular all root operators $\tilde{e}_{s}$, $\tilde{f}_{t}$ with $s,t \in \tilde{I}_{1}$ commute with the operators $\tilde{e}_{u}$, $\tilde{f}_{v}$, with $u,v \in \tilde{I}_{2}$. Moreover $\gamma_{i_{j}} = 1$ for all $j \in [1,r]$. Altogether, this implies: 

\begin{align*}
\tilde{\xi}_{\sigma(I)} (\Psi(b))  &=   \xi^{Y}_{\tilde{I}_{1}}  \xi^{Y}_{\tilde{I}_{2}} ({f}^{v}_{{i_{r}}} \cdots {f}^{v}_{{i_{1}}} (\Psi(\pi_{\nu}))) \\
                                                &=    \xi^{Y}_{\tilde{I}_{1}}  \xi^{Y}_{\tilde{I}_{2}} (\tilde{f}_{{i^{1}_{r}}}\tilde{f}_{{i^{2}_{r}}} \cdots \tilde{f}_{{i^{1}_{1}}} \tilde{f}_{{i^{2}_{1}}}(\Psi(\pi_{\nu}))) \\
                                                &=    \xi^{Y}_{\tilde{I}_{1}}  \xi^{Y}_{\tilde{I}_{2}} (\tilde{f}_{{i^{2}_{r}}}\cdots \tilde{f}_{{i^{2}_{1}}}  \tilde{f}_{{i^{1}_{r}}} \cdots \tilde{f}_{{i^{1}_{1}}}(\Psi(\pi_{\nu})))\\ 
                                                &=    \xi^{Y}_{\tilde{I}_{1}}  (\tilde{e}_{\theta_{\tilde{I}_{2}}({i^{2}_{r}})}\cdots \tilde{e}_{\theta_{\tilde{I}_{2}}({i^{2}_{1}})}  \tilde{f}_{{i^{1}_{r}}} \cdots \tilde{f}_{{i^{1}_{1}}}(\Psi(\pi_{\nu})))     \\                
                                                &=    \xi^{Y}_{\tilde{I}_{1}}  ( \tilde{f}_{{i^{1}_{r}}} \cdots \tilde{f}_{{i^{1}_{1}}}\tilde{e}_{\theta_{\tilde{I}_{2}}({i^{2}_{r}})}\cdots \tilde{e}_{\theta_{\tilde{I}_{2}}({i^{2}_{1}})} (\Psi(\pi_{\nu})))      \\    
                                                &=     \tilde{e}_{\theta_{\tilde{I}_{1}}({i^{1}_{r}})} \cdots \tilde{e}_{\theta_{\tilde{I}}({i^{1}_{1}})}\tilde{e}_{\theta_{\tilde{I}_{2}}({i^{2}_{r}})}\cdots \tilde{e}_{\theta_{\tilde{I}_{2}}({i^{2}_{1}})} (\Psi(\pi_{\nu})) \\
                                                &=     \tilde{e}_{\theta_{\tilde{I}_{1}}({i^{1}_{r}})} \tilde{e}_{\theta_{\tilde{I}_{2}}({i^{2}_{r}})} \cdots \tilde{e}_{\theta_{\tilde{I}_{1}}({i^{1}_{1}})}\tilde{e}_{\theta_{\tilde{I}_{2}}({i^{2}_{1}})} (\Psi(\pi_{\nu})).
\end{align*}
\end{enumerate}
 The case $\theta_{Y} = \operatorname{Id}$ is very similar.
\end{proof}

\begin{corollary}
The virtual cactus group $J^{v}_{X}$ acts on $\mathcal{P}(\psi(\lambda))$ and preserves the image $\Psi(\mathcal{P}(\lambda))$ of $\Psi$. 
\end{corollary}


\end{document}